\documentclass{amsart}
\usepackage[sorted-cites]{amsrefs}
\usepackage{graphicx}
\usepackage{latexsym}
\usepackage{amsthm}

\usepackage{amsfonts}
\usepackage{xcolor}
\usepackage[all]{xy}
\usepackage{dsfont}

\usepackage{amssymb, mathrsfs, amsfonts, amsmath}
\usepackage{amsbsy}
\usepackage{amsfonts}

\newtheorem{corollary}{Corollary}

\newtheorem{lemma}{Lemma}
\newtheorem{proposition}{Proposition}
\newtheorem{remark}{Remark}
\newtheorem{theorem}{Theorem}

\numberwithin{equation}{section}

\makeatletter
\@namedef{subjclassname@2020}{%
  \textup{2020} Mathematics Subject Classification}
\makeatother

\title[Critical metrics of the volume functional]{Spectral properties for critical metrics\\ of the volume functional}

\author{Rafael Di\'ogenes}
\author{Jaciane Gonçalves}
\author{Ernani Ribeiro Jr}

		\address[R. Di\'ogenes]{UNILAB, Instituto de Ci\^encias Exatas e da Natureza, Rua Jos\'e Franco de Oliveira, s/n, 62790-970, Reden\c{c}\~ao - CE, Brazil}\email{rafaeldiogenes@unilab.edu.br}

			\address[J. Gonçalves]{Universidade Estadual do Cear\'a - UECE, Centro de Ci\^encias e Tecnologia, Itaperi, 60714-903, Fortaleza - CE, Brazil}
	\email{jaciane.goncalves@uece.br}

		\address[E. Ribeiro]{Universidade Federal do Cear\'a - UFC, Departamento  de Matem\'atica, Campus do Pici, Av. Humberto Monte, Bloco 914, 60455-760, Fortaleza - CE, Brazil}
		\email{ernani@mat.ufc.br} 
	
\thanks{R. Di\'ogenes was partially supported by CNPq/Brazil - 305731/2024-6}

	\thanks{E. Ribeiro Jr was partially supported by CNPq/Brazil - 305128/2025-6 and 351492/2025-9}



\date{\today}

\begin{document}

\begin{abstract}
In this article, we investigate spectral properties of compact $V$-static manifolds, namely, compact manifolds with boundary whose metrics are critical points of the volume functional under a scalar curvature constraint. We derive sharp estimates for the first Steklov eigenvalue and the entire fourth-order Steklov spectrum. We further obtain a Lichnerowicz-type lower bound for the first eigenvalue of the drifted Laplacian naturally associated with the $V$-static potential. In the corresponding equality cases, we obtain rigidity results characterizing the Euclidean ball and the hemisphere.
\end{abstract}

\maketitle
\section{Introduction}
\label{intro}

A compact $n$-dimensional Riemannian manifold $(M^n,\,g),$ $n\geq 2,$ possibly with boundary $\partial M,$ is called a $V$-{\it static metric} if there exist a nonconstant smooth function $f$ on $M^n$ and a constant $\kappa$ satisfying the overdetermined elliptic system
\begin{equation}
\label{eqVS}
\left\{%
\begin{array}{lll}
    \displaystyle \mathcal{L}_g^{*}(f) =-(\Delta f)g + \nabla^2 f-fRic= \kappa g & \hbox{in $M,$} \\    
    \displaystyle f>0 & \hbox{on $int(M),$} \\
        \displaystyle f=0 & \hbox{on $\partial M,$} \\
\end{array}%
\right.
\end{equation} where $\mathcal{L}_g^{*}$ denotes the formal  $L^2$-adjoint of the linearization of the scalar curvature operator $\mathfrak{L}_g,$ $Ric$ is the Ricci tensor of  $(M^n,\,g),$ and $\Delta$ and $\nabla^2$ denote Laplacian and Hessian operators, respectively (cf.   \cites{corv2013def,miao2011einstein,miao2009volume}).

The study of $V$-static metrics arose from the investigation of critical metrics of the volume functional on compact manifolds with boundary. Motivated by the variational characterization of Einstein metrics as critical points of the Einstein--Hilbert functional \cite{besse2007einstein}*{Theorem 4.21}, together with the volume comparison theorem of Fan, Shi, and Tam \cites{fan2007large}, Miao and Tam \cites{miao2011einstein,miao2009volume} and Corvino, Eichmair, and Miao \cite{corv2013def} initiated the study of critical points of the volume functional restricted to the space of metrics with prescribed constant scalar curvature and fixed boundary metric. In this setting, $V$-static metrics are precisely the corresponding critical metrics; see, e.g., \cite{corv2013def}*{Theorem 2.3} and \cite{miao2009volume}*{Theorem 2.1}.

An important structural property of $V$-static manifolds is that every connected solution of \eqref{eqVS} has constant scalar curvature (cf. \cite{corv2013def}*{Proposition 2.1} and \cite{miao2009volume}*{Theorem 7}). Furthermore, 
if one considers the tensor equation \eqref{eqVS} independently of the boundary normalization, constant potentials correspond to Einstein metrics. Another distinguished subclass arises when $\kappa=0,$ in which case \eqref{eqVS} becomes the static vacuum system; see \cites{corvino2000scalar,corv2013def,HE1975}. Recently, McCormick \cite{McCormick} observed that $V$-static metrics also appear in asymptotically hyperbolic geometry as critical points of the volume-renormalized mass.
The geometric significance of $V$-static metrics was further highlighted by Corvino, Eichmair, and Miao \cite{corv2013def}, who established a deformation theorem showing that scalar curvature alone is insufficient to guarantee volume comparison. 
This result is particularly related to Schoen's conjecture on scalar curvature deformation \cite{Schoen2}; see also \cites{corv2013def,yuan}.

Explicit examples of $V$-static metrics on both compact and noncompact manifolds were constructed in \cites{corv2013def,miao2011einstein,miao2009volume}, including the spatial Schwarzschild and AdS--Schwarzschild metrics restricted to suitable domains containing the horizon and bounded by two spherically symmetric hypersurfaces, as well as geodesic balls in simply connected space forms endowed with their standard metrics. The study of $V$-static metrics has attracted considerable attention in recent years, leading to a variety of rigidity, classification, and geometric inequality results that deepen our understanding of the relationship between scalar curvature, volume, and boundary geometry; see, e.g., \cites{Maria,barros2015bach,BDR21,BDRIJM,Batistaetal2017,Allan2,corv2013def,DPRIsrael,FY,He1,KS,miao2011einstein,miao2009volume,yuan} and the references therein.

Alongside rigidity and classification results, spectral methods provide a powerful framework for understanding the geometry of Riemannian manifolds. Since spectra encode important geometric information, eigenvalue estimates have long been a fundamental topic in geometric analysis. A classical result in this direction is the Lichnerowicz estimate \cite{Lich}, which states that if a closed Riemannian manifold satisfies $Ric\geq (n-1)k \,g$ for some positive constant $k,$ then the first nonzero eigenvalue of the Laplacian operator is bigger than or equal to $n k.$ Moreover, Obata \cite{Obata} proved that equality holds if and only if the manifold is isometric to the standard sphere $\mathbb{S}^n,$ providing one of the fundamental rigidity theorems in spectral geometry. These results have stimulated numerous studies on sharp eigenvalue estimates under geometric assumptions. A significant advance was made by Bakry and \'Emery \cite{bakry}, who connected smooth metric measure spaces with functional inequalities and spectral theory; see also, e.g., \cites{BenNi,CZ2017,CZ2018}.

Despite substantial progress in rigidity and classification, the spectral properties of \(V\)-static manifolds remain comparatively less understood. Since \(V\)-static metrics arise as critical points of the volume functional, their distinguished geometry is expected to impose significant restrictions on the spectra of associated differential operators. A first indication of this interaction is obtained by pairing the traced \(V\)-static equation with a first Dirichlet eigenfunction, which yields an exact identity for the first Dirichlet eigenvalue of the Laplace operator on compact \(V\)-static manifolds with boundary (see Proposition~\ref{pthmK1}). This observation leads us to investigate further boundary spectral problems in the \(V\)-static setting, beginning with the Steklov problem.

Given a compact Riemannian manifold $(M^n,\,g)$ with boundary $\partial M$, the Steklov (or Stekloff) problem \cite{Stekloff} consists in finding non-trivial solutions of
\begin{equation}\label{stekloffproblem}
\Delta u = 0 \ \text{in } M^n, \qquad \frac{\partial u}{\partial \nu} = p\, u \ \,\,\,\,\text{on }\, \partial M,
\end{equation}
where $p\in \mathbb R.$ The Steklov eigenvalue problem has connections with
inverse problems, elasticity, fluid mechanics, heat transmission, and vibration theory;
see, e.g., \cite{GP2017,WANGXIA} and the re\-fe\-rences therein. Moreover, the Steklov eigenvalues coincide with the
spectrum of the Dirichlet-to-Neumann map. Its spectrum is discrete,
nonnegative, and unbounded, and, when $M$ is connected, can be arranged as
\[
0=p_0<p_1\leq p_2\leq \cdots \to +\infty,
\]
where the eigenvalues are repeated according to multiplicity. 

The first nonzero Steklov eigenvalue, denoted by $p_1$, carries
significant geometric information. Some sharp estimates relating $p_1$ to the geometry of the boundary $\partial M$ were obtained in, e.g., \cites{AlessandriniMagnanini1994,CGH2020,Escobar1997,Escobar1999,Escobar2000,FS2019,Kuttler1972,Payne1970,XX24}. In this context, Xia and Wang \cite{WANGXIA}, using the classical Reilly formula \cite{Reilly2}, proved that if $(M^n,g)$ has nonnegative Ricci curvature and the principal curvatures of $\partial M$ are bounded from below by a positive
constant $c$, then
\begin{equation}\label{estp_1wangxia}
    p_1 \leq \frac{\sqrt{\lambda_1^{\partial M}}}{(n-1)c}\left(\sqrt{\lambda_1^{\partial M}} + \sqrt{\lambda_1^{\partial M} - (n-1)c^2}\right),
\end{equation}
where $\lambda_1^{\partial M}$ denotes the first non-zero eigenvalue of the Laplacian on $\partial M$. Moreover, equality holds if and only if $M$ is isometric to a Euclidean ball of radius $1/c$. The proof of (\ref{estp_1wangxia}) relies on an estimate for $\lambda_1^{\partial M}$ established in \cite{CXIA}.

Inspired by the work of Xia--Wang \cite{WANGXIA}, Miao--Tam \cites{miao2011einstein,miao2009volume}, Batista et al. \cite{Batistaetal2017}, and Colbois et al. \cites{CAG2011,CGH2020}, we shall use the recent generalized Reilly's formula due to Qiu and Xia \cite{Qiu-Xia} (see also \cite{LiXia}) to establish a sharp geometric bound for the first nonzero eigenvalue $p_1$ of the Steklov problem on a compact $V$-static manifold with boundary.  More precisely, we have the following result.

\begin{theorem}
\label{thmB}
Let $(M^n, g, f),$ $n\geq 3,$ be an $n$-dimensional compact $V$-static manifold with connected boundary $\partial M,$ $\kappa\neq 0 $ and nonnegative scalar curvature. Then 
\begin{equation}
\label{ineqThmB}
p_1 \geq \frac{H}{n-1}.
\end{equation}
Moreover, equality in (\ref{ineqThmB}) holds  if and only if $(M^n,\,g)$ is isometric to a Euclidean ball of radius $1/p_1.$
\end{theorem}

As a consequence of Theorem~\ref{thmB}, and in the spirit of \cite{CAG2011}, we obtain the following sharp geometric bound for the Steklov eigenvalue $p_1,$ in which the condition $\kappa\neq 0$ is no longer required.

\begin{corollary}
    \label{corthmB}
    Let $(M^n, g, f),$ $n\geq 3,$ be an $n$-dimensional compact scalar flat $V$-static manifold with connected boundary $\partial M.$ Then 

    \begin{equation}
    \label{ineqcorthmB}
        p_1 \geq \frac{|\partial M|}{n\,Vol(M).}
    \end{equation} Moreover, equality in (\ref{ineqcorthmB}) holds if and only if $(M^n,\,g)$ is isometric to a Euclidean ball.
\end{corollary}

In this direction, and also motivated by \cites{BFG2009,WANGXIA}, it is natural to consider the fourth-order Steklov eigenvalue problem:

\begin{equation}
\label{Stek4}
\begin{cases}
\Delta^2 u=0 & \text{in } M,\\
u=0 & \text{on } \partial M,\\
\Delta u=q\,\dfrac{\partial u}{\partial\nu}
& \text{on } \partial M
\end{cases}
\end{equation} where $q\in \mathbb{R}.$ As observed in \cite{WANGXIA}, this problem has a natural connection with the theory of elasticity (cf. \cite{BFG2009}) and, as we shall see later, it is particularly well suited to the setting of $V$-static manifolds with boundary. 

This interaction leads to the following result.

\begin{theorem}
\label{thmD}
Let $(M^n,g,f)$, $n\geq 3$, be a compact $V$-static manifold with connected boundary $\partial M,$ $\kappa\neq 0$ and nonnegative scalar curvature. Let $q_1$ denote the first eigenvalue of the fourth-order Steklov problem (\ref{Stek4}). Then
\begin{equation}
\label{eq00athmD}
    q_1\geq \frac{n}{n-1}H.
\end{equation} Equality in (\ref{eq00athmD}) holds if and only if $R=0.$ Moreover, in the equality case, every first eigenfunction is a nonzero constant multiple of the $V$-static potential $f.$

\end{theorem} 

In the scalar flat case, integrating \eqref{eqlap} over $M$ and using 
the divergence theorem together with \eqref{ek89011thmD}, we obtain the identity
\begin{equation}\label{eq1athmD}
    \frac{|\partial M|}{\operatorname{Vol}(M)}
    =\frac{n}{n-1}H
    =q_1.
\end{equation}
Motivated by the spectral comparisons between different boundary value 
problems established by Kuttler and Sigillito \cite{KS1968}, and by the 
related question raised in \cite[Question~4.1]{CAG2011}, we combine 
Theorem~\ref{thmD} with Corollary~\ref{corthmB} to obtain the following 
sharp comparison.

\begin{corollary}
Let $(M^n,g,f)$, $n\geq 3$, be a compact scalar-flat $V$-static manifold 
with connected boundary $\partial M$. Then
\begin{equation}
    n p_1 \geq q_1.
\end{equation}
Moreover, equality holds if and only if $(M^n,g)$ is isometric to a 
Euclidean ball.
\end{corollary}

Next, as a consequence of the proof of Theorem \ref{thmD} and \cite{corv2013def}*{Proposition 2.5}, we have the following corollary.
\begin{corollary}
    \label{corthm3_a}
    Let $(M^n,g,f)$, $n\geq 3$, be a compact scalar flat $V$-static manifold with connected boundary $\partial M.$ Then 
    \begin{equation}
\label{ineqThmD}
q_1^2
\leq
\frac{n^2}{(n-1)(n-2)|\partial M|}
\int_{\partial M}R^{\partial M}\,dS_g,
\end{equation}
where $R^{\partial M}$ denotes the scalar curvature of the induced
metric on $\partial M$. Equality in \eqref{ineqThmD} holds if and only
if $(M^n,g)$ is isometric to a Euclidean ball.

\end{corollary}

An advantage of our approach is that it is not restricted to the
first eigenvalue. By combining a Riemannian version of Fichera's duality
principle (Lemma~\ref{lem:Fichera}) with the scalar flat $V$-static equation,
we obtain an exact relation between the entire fourth-order Steklov spectrum
and a weighted variational spectrum on the space of harmonic functions. In this sense, we denote by $\mathcal H(M)$ the completion of $\left\{
h\in C^\infty(\overline M):\Delta h=0 \text{ in }M
\right\}$ with respect to the $L^2(\partial M)$-norm,
and we let $E$ range over the $k$-dimensional linear
subspaces of $\mathcal H(M)$. Thus, we have the following result. 

%In the scalar flat $V$-static setting, the weighted energy \[ \int_M f|\nabla h|^2\,dV_g \] is understood by continuous extension to $\mathcal H(M)$; see Section~2.2. 

\begin{theorem}
\label{ThmD_k}
Let $(M^n,\,g,\,f)$, $n\geq 3$, be a compact scalar flat $V$-static
manifold with connected boundary $\partial M$. Let $q_k$ denote the $k$-th eigenvalue of the fourth-order Steklov problem (\ref{Stek4}). Then, for every
$k\geq1,$ 
\begin{equation}
\label{eq:qk-gammak}
q_k
=
\frac{|\partial M|}{Vol(M)}
+
\frac{2}{|\nabla f|_{\partial M}}\,\left(\inf_{\substack{E\subset\mathcal H(M)\\ \dim E=k}}
\;
\sup_{0\neq h\in E}
\frac{\displaystyle\int_M f|\nabla h|^2\,dV_g}
{\displaystyle\int_M h^2\,dV_g}\right).
\end{equation}
\end{theorem}

We now turn to the drifted Laplacian naturally associated with the $V$-static potential. Since $f>0$ in $int(M)$ and $f=0$ on $\partial M,$ the function $h=-\ln f $ is defined in $int(M).$ Thus, the corresponding drifted Laplacian can be written in divergence form as
\[
L_f w:=\Delta_{-\ln f}w
=
\Delta w+\frac{1}{f}\langle \nabla f,\nabla w\rangle
=
\frac{1}{f}\operatorname{div}(f \nabla w)
\]
in $int(M)$. So, the natural
measure associated with $L_f$ is  $e^{-h}dV_g = f \, dV_g.$ Accordingly, instead of imposing a Dirichlet boundary condition, we
consider the natural weighted spectral problem associated with the
quadratic form
\[
\mathcal{Q}_f(w)
=
\int_M f|\nabla w|^2\,dV_g.
\]
Since constant functions belong to the kernel of $L_f$,
we denote by $\mu_1$ the first nonzero weighted eigenvalue, characterized
variationally by
\[
\mu_1
=
\inf_{\substack{w\not\equiv 0\\
\int_M wf\,dV_g=0}}
\frac{\displaystyle\int_M f|\nabla w|^2\,dV_g}
{\displaystyle\int_M fw^2\,dV_g}.
\]
Equivalently, an eigenfunction associated with $\mu_1$ satisfies, in the
weak sense,
\[
\operatorname{div}(f\nabla w)
=
-\mu_1fw\,\,\,\,\,\,\hbox{and}\,\,\,\,\,
\int_M wf\,dV_g=0.
\]
Moreover, since $f=0$ on $\partial M$, the boundary term arising from
integration by parts vanishes naturally, so that no additional boundary
condition is imposed in this weighted formulation.

Based on this discussion, we have established the following sharp Lichnerowicz-type estimate for the first nonzero eigenvalue $\mu_1$ of the drifted Laplacian $L_f.$

\begin{theorem}\label{thmC}
Let $(M^n,g,f)$ be an $n$-dimensional compact $V$-static manifold with
connected boundary $\partial M$ and $\kappa\geq 0$. Then 
\begin{equation}\label{thmC-estimate}
\mu_1
\geq
\frac{n+1}{n(n-1)}
\left(
R+\frac{\kappa}{f_{\max}}
\right).
\end{equation}
Moreover, equality in \eqref{thmC-estimate} holds if and only if
$\kappa=0$ and $(M^n,\,g)$ is isometric to a hemisphere $\mathbb{S}^{n}_{+}
\left(
\sqrt{n(n-1)/R}
\right).$
\end{theorem}

As a consequence of Theorem \ref{thmC}, we obtain the following corollary.

\begin{corollary}\label{cor-static-weighted}
Let $(M^n,\,g,\,f)$ be an $n$-dimensional compact static manifold with
connected boundary $\partial M.$ Then 
\begin{equation}\label{cor-static-weighted-eq}
\mu_1
\geq
\frac{n+1}{n}\lambda_1.
\end{equation}
Moreover, equality in \eqref{cor-static-weighted-eq} holds if and only if
$(M^n,g)$ is isometric to a hemisphere $\mathbb{S}^{n}_{+}
\left(
\sqrt{n/\lambda_1}
\right).$
\end{corollary}

\vspace{0.30cm}

The rest of this paper is organized as follows: In Section \ref{prelim}, we review the main background material, preliminary results on $V$-static metrics and key lemmas that will be used in the sequel. Section \ref{SecProofs} is devoted to the proofs of the main results.
 
\section{Background}
\label{prelim}
In this section, we review some basic facts and present some key results that will play a crucial role in the proof of the main theorems.

We begin recalling that a compact Riemannian manifold $(M^n,\,g)$ with boundary $\partial M$ is a $V$-static metric if there exists a nonnegative smooth function $f$ satisfying 
\begin{equation}\label{eqdefsec2}
    -(\Delta f)g + \nabla^2f - fRic =  \kappa g,
\end{equation} where $\kappa$ is a constant and $f^{-1}(0)=\partial M.$ 

Taking the trace in (\ref{eqdefsec2}), one sees that 
\begin{equation}\label{eqlap}
    \Delta f = -\frac{fR +  \kappa n}{n-1}.
\end{equation} 
From this, we have
\begin{equation}\label{eq2.3}
    \nabla^2f-fRic=-\frac{Rf+\kappa}{n-1}g
\end{equation}
and 
\begin{equation}\label{equivRicHess}
    f\mathring{Ric} = \mathring{\nabla^2 f},
\end{equation} where $\mathring{Z}=Z-\frac{{\rm tr}\,Z}{n}g$ denotes the traceless part of the tensor $Z.$ 

It should be mentioned that, choosing appropriate coordinates, $f$ and $g$ are analytic; see Proposition 2.1 in \cite{corv2013def}. Consequently, the set of regular points of $f$ is dense in $M^n.$ Besides, at every regular point of $f,$ the vector field $\nu=-\frac{\nabla f}{|\nabla f|}$ is normal to the corresponding level hypersurface. In particular, along $\partial M,$ it coincides with the outward unit normal. Also, it is known from  \cite[Theorem 7]{miao2009volume} that $|\nabla f|$ is constant (non null) on each connected component of $\partial M$. Furthermore, the second fundamental form of $\partial M$ is given by
\begin{equation*}
    \mathbb{I}\mathbb{I}(e_i, e_j)=\langle \nabla_{e_i}\nu, e_j\rangle,
\end{equation*}
where $\{e_1,...,e_{n-1}\}$ is an orthonormal frame on $\partial M$. Thus, one obtains from (\ref{eq2.3}) that
\begin{equation}\label{segform}
    \mathbb{I}\mathbb{I}(e_i, e_j)=-\left \langle \nabla_{e_i}\frac{\nabla f}{|\nabla f|}, e_j\right \rangle=\frac{\kappa}{(n-1)|\nabla f|}g_{ij}.
\end{equation}
Consequently, the mean  curvature is constant $H=\frac{\kappa}{|\nabla f|}$ on each connected component of $\partial M.$ Therefore, $\partial M$ is totally umbilical and when $\kappa=0$, $\partial M$ is totally geodesic. 

Also, it follows from \cite{Batistaetal2017}*{Lemma 6} that

$$
R^{\partial M}_{ij}
=
R_{ij}-R_{injn}
+
\frac{(n-2)\kappa^2}
{(n-1)^2|\nabla f|^2}g_{ij},
$$ where $\{e_1,\ldots,e_{n-1}\}$ is an orthonormal frame tangent to
$\partial M$. In particular, taking the trace, we obtain
\begin{equation}
\label{eqnhu850001a}
2R_{nn}+R^{\partial M}=R+\frac{(n-2)}{(n-1)}H^2;
\end{equation} see also \cite[Eq.~(45)]{miao2009volume}.

In the sequel, as mentioned in the Introduction, we derive an exact identity for the first Dirichlet eigenvalue of the Laplace operator on compact \(V\)-static manifolds with boundary.

\begin{proposition}
\label{pthmK1}
    Let $(M^n,\, g,\, f)$ be an $n$-dimensional compact $V$-static manifold with connected boundary $\partial M$ and let $\lambda_1$ denote the first Dirichlet eigenvalue of the Laplace operator on $M^n.$ If $\varphi_1$ is a positive first Dirichlet eigenfunction, then
\begin{equation}
\label{key1eqthmk1}
\lambda_1 = \frac{R}{n-1}
+
\frac{n\kappa}{n-1}
\left(\frac{\displaystyle\int_M \varphi_1\,dV_g}
{\displaystyle\int_M f\varphi_1\,dV_g}\right)
.\end{equation}
\end{proposition}
\begin{proof}
    Let $\varphi_1$ be a first Dirichlet eigenfunction of the Laplacian operator so that
$\varphi_1>0$ in $int(M).$ Thus, we have
\begin{equation}
\label{eqkj76990}
\Delta\varphi_1=-\lambda_1\varphi_1
\,\,\,\, \text{in } M\,\,\,\,\hbox{and}\,\,\,\,
\varphi_1=0
\,\,\,\, \text{on } \partial M.
\end{equation} Moreover, by (\ref{eqlap}), one sees that
\begin{equation}
\label{eqlap1111100}
\Delta f=-\frac{Rf+n\kappa}{n-1}.
\end{equation}

Since $f=0$ and 
$\varphi_1=0$ on $\partial M,$ we use Green's identity to infer
\[
\int_M
\left(
f\Delta\varphi_1-\varphi_1\Delta f
\right)dV_g
=
\int_{\partial M}
\left(
f\frac{\partial\varphi_1}{\partial\nu}
-
\varphi_1\frac{\partial f}{\partial\nu}
\right)dS_g
=0,
\]
that is,
\[
\int_M f\Delta\varphi_1\,dV_g
=
\int_M\varphi_1\Delta f\,dV_g.
\] Whence, it follows from (\ref{eqkj76990}) and (\ref{eqlap1111100})  that

\[
-\lambda_1
\int_M f\varphi_1\,dV_g
=
-\frac{R}{n-1}
\int_M f\varphi_1\,dV_g
-
\frac{n\kappa}{n-1}
\int_M\varphi_1\,dV_g,
\] so that

\[
\left(
\lambda_1-\frac{R}{n-1}
\right)
\int_M f\varphi_1\,dV_g
=
\frac{n\kappa}{n-1}
\int_M\varphi_1\,dV_g.
\] Taking into account that $f>0$ and $\varphi_1>0$ in $int(M)$, we achieve
\[
\lambda_1
=
\frac{R}{n-1}
+
\frac{n\kappa}{n-1}
\frac{\displaystyle\int_M\varphi_1\,dV_g}
{\displaystyle\int_M f\varphi_1\,dV_g},
\] which finishes the proof. 
\end{proof}

\begin{remark}
We note that Proposition \ref{pthmK1} shows that the position of the first Dirichlet eigenvalue relative to $\frac{R}{n-1}$ is completely determined by the sign of $\kappa$. In particular, for $\kappa>0,$ (\ref{key1eqthmk1}) implies
\begin{equation}
    \lambda_1 > \frac{R}{n-1}
+
\frac{n\kappa}{(n-1)f_{max}},
\end{equation} where $f_{max}$ is the maximum value of $f$ on $M^n.$
\end{remark}

\subsection{Generalized Reilly's formula} 
Reilly's formula \cite{Reilly2} has proved to be a powerful tool for deriving several geometric inequalities and rigidity results. For instance, Ros \cite{Ros} used it to establish an Alexandrov-type rigidity theorem for higher-order mean curvatures. Miao, Tam, and Xie \cite{MTX} applied Reilly's formula to derive a stability inequality for the Wang--Yau energy, while Kwong and Miao \cite{KW1} obtained a related inequality for boundaries of static spaces. Subsequently, Qiu and Xia \cite{Qiu-Xia} established a generalized Reilly formula, which they used to give an alternative proof of Alexandrov's theorem and to derive a new Heintze--Karcher-type inequality. 

The following generalized Reilly's formula due to Qiu and Xia \cite{Qiu-Xia} will be very useful (see also \cite{LiXia}).

 \begin{proposition}[\cite{Qiu-Xia}]\label{prop-quiu-xia}
Let $(M^n,\,g)$ be an $n$-dimensional, compact Riemannian manifold with boundary $\partial M.$ Given two functions $f,$ $u$ on $M$ and a constant $k,$ one has 

\begin{eqnarray*}
&&\int_{M} f \left[(\Delta u + knu)^{2}-|\nabla^{2}u+kug|^{2}\right]dV_g=(n-1)k\int_{M}(\Delta f +nkf)u^{2}dV_g \nonumber \\ &&+\int_{M}\left[\nabla ^{2}f-(\Delta f)g-2(n-1)kfg+fRic\right](\nabla u, \nabla u)dV_g \nonumber \\
&&+\int_{\partial M}f \left[2\left(\frac{\partial u}{\partial \nu}\right)\Delta_{\partial M}u+H\left(\frac{\partial u}{\partial \nu}\right)^{2}+\mathbb{II}(\nabla_{\partial M}u, \nabla_{\partial M} u)+2(n-1)k\left(\frac{\partial u}{\partial \nu}\right)u\right]dS_g\\
&&+ \int_{\partial M}\frac{\partial f}{\partial \nu}\left(|\nabla _{\partial M}u|^{2}-(n-1)ku^{2}\right)dS_g, \nonumber
\end{eqnarray*}
where $\mathbb{II}$ and $H=tr(\mathbb{II})$ stand for the second fundamental form and the mean curvature of $\partial M$, respectively.
 \end{proposition}

As an application of the generalized Reilly's formula, we obtain the following key result.

\begin{proposition}
\label{prop-harmonic}
Let $(M^n,\,g,\,f)$ be an $n$-dimensional compact $V$-static manifold with
connected boundary $\partial M$, and let $u$ be a harmonic function on $M$.
Then
\begin{align*}
\int_M f|\nabla^2u|^2\,dV_g
={}&
-\frac{\kappa}{n-1}\int_M |\nabla u|^2\,dV_g
-\frac{R}{n-1}\int_M f|\nabla u|^2\,dV_g \\
&\quad
+c\int_{\partial M}
\left(\frac{\partial u}{\partial \nu}\right)^2\,dS_g,
\end{align*}
where $\nu$ is the outward unit normal to $\partial M$ and $c:=|\nabla f|_{\partial M}.$
\end{proposition}

\begin{proof}
Applying Reilly's formula (Proposition \ref{prop-quiu-xia}) considering $k=0,$ $\Delta u=0$ and $f=0$ on $\partial M,$ we obtain
\begin{align}
-\int_M f|\nabla^2u|^2\,dV_g
={}&
\int_M
\left(
\nabla^2f-(\Delta f)g+f Ric
\right)(\nabla u,\nabla u)\,dV_g
\nonumber\\
&\quad
+\int_{\partial M}
\frac{\partial f}{\partial \nu}
|\nabla_{\partial M}u|^2\,dS_g.
\label{prop-harmonic-eq1}
\end{align}

Since 
\begin{equation}
\label{eqrepetc}
\frac{\partial f}{\partial \nu}
=
\left\langle
\nabla f,-\frac{\nabla f}{|\nabla f|}
\right\rangle
=
-|\nabla f|
=
-c,
\end{equation} one sees from \eqref{prop-harmonic-eq1} that
\begin{align}
-\int_M f|\nabla^2u|^2\,dV_g
={}&
\kappa\int_M |\nabla u|^2\,dV_g
+
2\int_M
f Ric(\nabla u,\nabla u)\,dV_g
\nonumber\\
&\quad
-c\int_{\partial M}
|\nabla_{\partial M}u|^2\,dS_g.
\label{prop-harmonic-eq2}
\end{align}

Since $u$ is harmonic, Bochner's formula gives
\[
2f Ric(\nabla u,\nabla u)
=
f\Delta|\nabla u|^2
-
2f|\nabla^2u|^2.
\]
Integrating over $M$, we obtain
\begin{align*}
2\int_M
f Ric(\nabla u,\nabla u)\,dV_g
={}&
\int_M f\Delta|\nabla u|^2\,dV_g
-
2\int_M f|\nabla^2u|^2\,dV_g.
\end{align*}
By Green's formula,
\begin{equation*}
\int_M f\Delta|\nabla u|^2\,dV_g
=
\int_{\partial M}
\left(
f\frac{\partial |\nabla u|^2}{\partial \nu}
-
|\nabla u|^2\frac{\partial f}{\partial \nu}
\right)dS_g 
+\int_M |\nabla u|^2\Delta f\,dV_g.
\end{equation*}
Taking into account that $f=0$ on $\partial M$ and (\ref{eqrepetc}), one sees that
\[
\int_M f\Delta|\nabla u|^2\,dV_g
=
c\int_{\partial M}|\nabla u|^2\,dS_g
+
\int_M |\nabla u|^2\Delta f\,dV_g.
\]
By using (\ref{eqlap}), we get
\begin{align}
2\int_M
f Ric(\nabla u,\nabla u)\,dV_g
={}&
-\frac{R}{n-1}
\int_M f|\nabla u|^2\,dV_g
-
\frac{n\kappa}{n-1}
\int_M |\nabla u|^2\,dV_g
\nonumber\\
&\quad
+c\int_{\partial M}|\nabla u|^2\,dS_g
-
2\int_M f|\nabla^2u|^2\,dV_g.
\label{prop-harmonic-eq3}
\end{align}
Plugging \eqref{prop-harmonic-eq3} into
\eqref{prop-harmonic-eq2}, we obtain
\begin{align*}
-\int_M f|\nabla^2u|^2\,dV_g
={}&
\kappa\int_M |\nabla u|^2\,dV_g
-\frac{R}{n-1}
\int_M f|\nabla u|^2\,dV_g \\
&\quad
-\frac{n\kappa}{n-1}
\int_M |\nabla u|^2\,dV_g
+c\int_{\partial M}|\nabla u|^2\,dS_g \\
&\quad
-2\int_M f|\nabla^2u|^2\,dV_g
-c\int_{\partial M}
|\nabla_{\partial M}u|^2\,dS_g.
\end{align*}
Rearranging terms, one concludes that 
\begin{align}
\int_M f|\nabla^2u|^2\,dV_g
={}&
-\frac{\kappa}{n-1}
\int_M |\nabla u|^2\,dV_g
-\frac{R}{n-1}
\int_M f|\nabla u|^2\,dV_g
\nonumber\\
&\quad
+c\int_{\partial M}
\left(
|\nabla u|^2-|\nabla_{\partial M}u|^2
\right)dS_g.
\label{prop-harmonic-eq4}
\end{align}

Finally, along $\partial M$, it holds
\[
|\nabla u|^2-|\nabla_{\partial M}u|^2
=
\left(\frac{\partial u}{\partial \nu}\right)^2.
\]
Substituting this into \eqref{prop-harmonic-eq4}, one deduces that
\begin{eqnarray*}
\int_M f|\nabla^2u|^2\,dV_g
&=&
-\frac{\kappa}{n-1}\int_M |\nabla u|^2\,dV_g
-\frac{R}{n-1}\int_M f|\nabla u|^2\,dV_g
\nonumber\\&&+
c\int_{\partial M}
\left(\frac{\partial u}{\partial \nu}\right)^2\,dS_g.
\end{eqnarray*} So, the proof is completed. 
\end{proof}

\begin{remark}
If $\kappa\neq 0,$ then $H=\kappa/c.$ Thus, in the non-static case, the last term in the right hand side of Proposition \ref{prop-harmonic} can be expressed in terms of the mean curvature $H$ of the boundary $\partial M.$
\end{remark}

\subsection{Fichera-type duality} 
In \cite{Fichera1955}, Fichera introduced a general
duality principle in the Euclidean setting. Extensions to less regular domains were later obtained
by Bucur, Ferrero, and Gazzola \cite{BFG2009}, while the corresponding
duality principle for the entire spectrum was recently established by
Ferraresso and Lamberti \cite{FerraressoLamberti}; for more details, see \cites{BFG2009,FerraressoLamberti} and the re\-fe\-rences therein. As we shall show in the next
lemma, the same argument extends naturally to compact Riemannian manifolds
with smooth boundary, and this formulation will be used in the proof of
Theorem~\ref{ThmD_k}.

In order to proceed, let
\[
\mathcal H^\infty(M)
=
\left\{
h\in C^\infty(\overline M):
\Delta h=0 \text{ in }M
\right\}.
\]
We define $\mathcal H(M)$ as the completion of
$\mathcal H^\infty(M)$ with respect to the norm
\[
\|h\|_{\mathcal H(M)}
=
\left(
\int_{\partial M}h^2\,dS_g
\right)^{1/2}.
\]
For an element of $\mathcal H(M)$, its boundary value is understood as
the $L^2(\partial M)$-limit of any defining Cauchy sequence in
$\mathcal H^\infty(M)$. By the mapping properties of the Poisson operator on a compact
Riemannian manifold with smooth boundary, there exists a constant
$C>0$ such that
\begin{equation}\label{eq-harmonic-extension}
\|h\|_{H^{1/2}(M)}
\leq
C\|h\|_{L^2(\partial M)}
\end{equation}
for every $h\in\mathcal H^\infty(M)$; see, for instance,
\cite[Chapter~7]{TaylorPDEII}. Consequently, every element of
$\mathcal H(M)$ can be identified with a harmonic function in the
distributional sense, and the inclusion
\[
\mathcal H(M)\hookrightarrow L^2(M)
\]
is compact. Moreover, the harmonic Bergman space is given by
\[
L^2_{\mathcal H}(M)
=
\overline{\mathcal H(M)}^{\,L^2(M)}.
\] Thus, $\mathcal H(M)$ is a dense subspace of
$L^2_{\mathcal H}(M)$.

With this notation, we have the following Fichera-type duality.

\begin{lemma}[Fichera-type duality]
\label{lem:Fichera}
Let $(M^n,g)$ be a compact Riemannian manifold with smooth boundary
$\partial M$. A nontrivial function $u\in C^\infty(M)$ solves (\ref{Stek4})
if and only if the harmonic function $h=\Delta u$ is nontrivial and satisfies
\begin{equation}
\label{eq:Fichera-dual}
\int_{\partial M}h\psi\,dS_g
=
q\int_Mh\psi\,dV_g
\end{equation}
for every function $\psi\in \mathcal H(M)$.
\end{lemma}

\begin{proof}
    On $\mathcal H(M),$ consider the symmetric bilinear forms 
    \begin{equation}
       \label{formA}
        A(h,\psi)=\int_{\partial M} h\psi \,dS_g
    \end{equation} and

    \begin{equation}
    \label{formB}
        B(h,\psi)=\int_{M} h\psi\, dV_g.
    \end{equation} With this notation, the dual eigenvalue problem (\ref{eq:Fichera-dual}) becomes

\begin{equation}
\label{ABrelation}
    A(h,\psi)=q B(h,\psi)\,\,\,\,\hbox{for every}\,\,\psi\in \mathcal H(M).
\end{equation}

Suppose that $u$ is a nontrivial solution of (\ref{Stek4}) and $h=\Delta u.$ Since $\Delta^2 u=0$ in $M,$ we have $\Delta h=0,$ which shows that $h\in \mathcal H(M).$ Moreover, $h$ is nontrivial. Indeed, if $h\equiv 0,$ then $u$ would be harmonic and, since $u=0$ on $\partial M,$ the uniqueness of the Dirichlet problem would give $u\equiv 0,$ which is a contradiction. 

Next, for $\psi\in \mathcal H(M),$ Green's identity implies 

\begin{equation*}
    \int_{M} h\psi \,dV_g = \int_{M} (\Delta u)\psi \,dV_g= \int_{\partial M} \psi \frac{\partial u}{\partial \nu}\, dS_g.
\end{equation*} At the same time, the boundary condition in (\ref{Stek4}) yields

$$h=\Delta u=q \frac{\partial u}{\partial \nu}\,\,\,\,\,\hbox{on}\,\,\,\,\partial M,$$ and hence, 

\begin{equation*}
    A(h,\psi) = \int_{\partial M} h\psi \,dS_g=q\int_{\partial M} \psi\, \frac{\partial u}{\partial \nu}\,dS_g= q\int_{M}h\psi\,dV_g=qB(h,\psi),
\end{equation*} which proves (\ref{eq:Fichera-dual}).

Conversely, suppose that $0\neq h\in \mathcal H(M)$ satisfies (\ref{ABrelation}). Let $u$ be the unique solution of 

$$\Delta u=h\,\,\hbox{in}\,\, M\,\,\,\,\,\,\,\hbox{and}\,\,\,\,\,\,\,u=0\,\,\hbox{on}\,\,\partial M.$$ In particular, since $h$ is harmonic, we have $\Delta^2 u=0.$

Next, for any $\psi\in\mathcal H(M)$, Green's identity gives
\[
\int_M h\psi\,dV_g
=
\int_{\partial M}
\psi\frac{\partial u}{\partial\nu}\,dS_g.
\]
This substituted into
(\ref{ABrelation}) yields
\[
\int_{\partial M}
\left(
h-q\frac{\partial u}{\partial\nu}
\right)\psi\,dS_g
=0
\]
for every $\psi\in\mathcal H(M)$. Now, let $\varphi$ be an arbitrary smooth function on $\partial M$,
and consider $\psi$ to be its harmonic extension to $M$. Since $\psi|_{\partial M}=\varphi,$
we infer
\[
\int_{\partial M}
\left(
h-q\frac{\partial u}{\partial\nu}
\right)\varphi\,dS_g
=0
\]
for every $\varphi\in C^\infty(\partial M)$. Consequently, $h=q\frac{\partial u}{\partial\nu}$ on $\partial M.$
Recalling that $h=\Delta u$, we conclude that
\[
\Delta u
=
q\frac{\partial u}{\partial\nu}
\qquad\text{on }\partial M.
\]
Thus $u$ solves \eqref{Stek4}, and the equivalence is proved.
\end{proof}

\subsection{Bochner-type formula for $L_f$}

A complete smooth metric measure space $(M^n, g, e^{-h}dV_g)$ is a complete $n$-di\-men\-sio\-nal Riemannian manifold $(M^n, g)$ together with a weighted volume form $e^{-h}dV_g$ on $M^n$, where $h$ is a smooth function on $M^n$ and $dV_g$ the volume element induced by the metric $g$. A suitable operator on a smooth metric measure space is the drifted Laplacian operator 
\begin{align}\label{deflaplacianof}
    \Delta_h=\Delta -\langle \nabla h, \nabla\cdot\rangle,
\end{align}
where $\Delta$ is the Laplacian on $M^n.$

Choosing $h=-\ln f,$ the weighted measure becomes $e^{-h}dV_g = f \, dV_g$, and the drifted Laplacian can be written as
\begin{align}\label{deflapla-lnf}
L_f w:=\Delta_{- \ln f} w= \Delta w + \frac{1}{f} \langle \nabla  f, \nabla w \rangle.
\end{align} Notice that, for a compact $V$-static manifold with boundary, $L_f$ is smooth only in $int(M).$

Recall that, for $N>n$, the $N$-dimensional Bakry--\'Emery Ricci
tensor  \cite{bakry} of a weighted manifold $(M^n,g,e^{-h}dV_g)$ is defined by
\begin{equation*}
Ric_h^N
=
Ric+\nabla^2 h
-\frac{1}{N-n}dh\otimes dh.
\end{equation*} Taking $N=n+1$ and $h=-\ln f$, the $V$-static equation \eqref{eqVS} yields
\begin{equation}
Ric_h^{n+1}
=
\frac{1}{n-1}
\left(
R+\frac{\kappa}{f}
\right)g
\end{equation}
in $int (M).$ Notably, when $\kappa\geq0$, this
identity provides a natural lower bound for the $(n+1)$-dimensional
Bakry--\'Emery Ricci tensor.

The next lemma corresponds to a Bochner-type formula for $L_f.$ 

\begin{lemma}\label{lemma1}
    Let $(M^n,\, g)$ be a Riemannian manifold and $w, f \in C^{\infty}(M)$, with $f>0.$ Then
    
\begin{align*}
\frac{1}{2}L_f|\nabla w|^2=&|\nabla^2w|^2+Ric(\nabla w,\nabla w)-\frac{1}{f}\nabla^2f(\nabla w,\nabla w)\\
&+\langle\nabla (L_fw),\nabla w\rangle+\frac{1}{f^2}\langle\nabla w,\nabla f \rangle^2.
\end{align*}
\end{lemma}

\begin{proof}
By (\ref{deflapla-lnf}) and the classical Bochner's formula, we have
\begin{eqnarray}\label{lemAeq1}
\frac{1}{2}L_f|\nabla w|^2&=&\frac{1}{2}\Delta|\nabla w|^2+\frac{1}{2f}\langle \nabla |\nabla w|^2,\nabla f\rangle\nonumber\\
 &=&|\nabla^2 w|^2+Ric(\nabla w,\nabla w)+\langle\nabla\Delta w,\nabla w\rangle+\frac{1}{f}\nabla^2w(\nabla w,\nabla f).
\end{eqnarray}
Note that
\begin{eqnarray*}
\langle\nabla(L_f w),\nabla w\rangle&=&\langle\nabla(\Delta w+\frac{1}{f}\langle\nabla w,\nabla f\rangle),\nabla w\rangle\nonumber\\
 &=&\langle\nabla\Delta w,\nabla w\rangle+\frac{1}{f}\langle\nabla(\langle \nabla w,\nabla f\rangle),\nabla w\rangle-\frac{1}{f^2}\langle \nabla w,\nabla f\rangle^2\nonumber\\
 &=&\langle\nabla\Delta w,\nabla w\rangle+\frac{1}{f}\nabla^2w(\nabla w,\nabla f)\\
 &&+\frac{1}{f}\nabla^2f(\nabla w,\nabla w)-\frac{1}{f^2}\langle \nabla w,\nabla f\rangle^2.
\end{eqnarray*} This combined with (\ref{lemAeq1}) yields
\begin{align*}
\frac{1}{2}L_f|\nabla w|^2=&|\nabla^2w|^2+Ric(\nabla w,\nabla w)-\frac{1}{f}\nabla^2f(\nabla w,\nabla w)\\
&+\langle\nabla (L_f w),\nabla w\rangle+\frac{1}{f^2}\langle\nabla w,\nabla f \rangle^2, 
\end{align*} as stated. 
\end{proof}

The following algebraic lemma will be useful.

\begin{lemma}\label{lemma-tec}
Let $a,b : M \to \mathbb{R}$ be real-valued functions and let $\alpha>0$. Then
\begin{equation}
(a+b)^2\geq\frac{a^2}{1+\alpha}-\frac{b^2}{\alpha}\label{des1}
\end{equation}
holds on $M$. Moreover, equality holds in (\ref{des1}) if and only if $b=-\frac{\alpha}{1+\alpha}a$.
\begin{proof}
Notice that
\begin{eqnarray*}
0&\leq&\left(\sqrt{\frac{\alpha}{1+\alpha}}a+\sqrt{\frac{1+\alpha}{\alpha}}\,b\right)^2\\
 &=&\frac{\alpha}{1+\alpha}a^2+2ab+\frac{1+\alpha}{\alpha}b^2\\
 &=&(a+b)^2-\frac{a^2}{1+\alpha}+\frac{b^2}{\alpha},
\end{eqnarray*} which proves the stated inequality. 
\end{proof}
\end{lemma}

As an application of Lemmas \ref{lemma1} and \ref{lemma-tec}, we have the following result.

\begin{lemma}\label{lemB}
Let $(M^n,\,g,\,f)$ be a $V$-static manifold. Then, for every $w \in C^{\infty}(M),$ the following inequality holds in $int(M)$:
\begin{equation}\label{eqlemaB}
\frac{1}{2} L_f |\nabla w|^2\geq \frac{(L_f  w)^2}{(n+1)}+\frac{1}{n-1}\left(R+\frac{\kappa}{f}\right)|\nabla w|^2 +\langle\nabla ( L_f w),\nabla w\rangle.
\end{equation}
Moreover, equality in (\ref{eqlemaB}) holds if and only if  $\nabla^2 w = \frac{\Delta w}{n} g = \frac{\langle \nabla f, \nabla w\rangle}{f} g$.
\end{lemma}

\begin{proof}
Since $|\nabla^2w|^2\geq\frac{(\Delta w)^2}{n},$ we use Lemma \ref{lemma1} and (\ref{deflapla-lnf}) to obtain
\begin{eqnarray*}
\frac{1}{2} L_f |\nabla w|^2&\geq&\frac{(\Delta w)^2}{n}+Ric(\nabla w,\nabla w)-\frac{1}{f}\nabla^2f(\nabla w,\nabla w)\\
&&+\langle\nabla (L_f   w),\nabla w\rangle+\frac{1}{f^2}\langle\nabla f,\nabla w\rangle^2\\
 &=&\frac{(L_f w-\frac{1}{f}\langle\nabla w,\nabla f\rangle)^2}{n}+Ric(\nabla w,\nabla w)-\frac{1}{f}\nabla^2f(\nabla w,\nabla w)\nonumber\\
 &&+\langle\nabla(L_f w),\nabla w\rangle+\frac{1}{f^2}\langle\nabla f,\nabla w\rangle^2.
\end{eqnarray*}
By Lemma \ref{lemma-tec}, one deduces that
\begin{eqnarray*}
 \frac{1}{2}L_f |\nabla w|^2&\geq&\frac{(L_f w)^2}{(1+\alpha)n}-\frac{1}{\alpha nf^2}\langle\nabla w,\nabla f\rangle^2+Ric(\nabla w,\nabla w)\\
 &&-\frac{1}{f}\nabla^2f(\nabla w,\nabla w)\nonumber+\langle\nabla (L_f w),\nabla w\rangle+\frac{1}{f^2}\langle\nabla w,\nabla f\rangle^2.
\end{eqnarray*} Now, it suffices to take $\alpha=\frac{1}{n}$ in order to infer

\begin{eqnarray}
\label{qwe45ggghj8}
\frac{1}{2} L_f |\nabla w|^2 &\geq &\frac{(L_f  w)^2}{(n+1)}+Ric(\nabla w,\nabla w)-\frac{1}{f}\nabla^2f(\nabla w,\nabla w)\nonumber\\
&& +\langle\nabla ( L_f w),\nabla w\rangle.
\end{eqnarray}

On the other hand, it follows from (\ref{eq2.3}) that

\begin{eqnarray*}
    Ric - \frac{1}{f}\nabla^2 f=\frac{1}{n-1}\left( R+\frac{\kappa}{f}\right)g
\end{eqnarray*} in $int(M),$ and hence,

\begin{equation*}
    Ric(\nabla w, \nabla w)-\frac{1}{f}\nabla^2 f(\nabla w,\nabla w)=\frac{1}{n-1}\left(R+\frac{\kappa}{f}\right)|\nabla w|^2.
\end{equation*} Substituting this into (\ref{qwe45ggghj8}) yields

$$\frac{1}{2} L_f |\nabla w|^2\geq \frac{(L_f  w)^2}{(n+1)}+\frac{1}{n-1}\left(R+\frac{\kappa}{f}\right)|\nabla w|^2 +\langle\nabla ( L_f w),\nabla w\rangle,$$ which proves (\ref{eqlemaB}).

We now address the equality case. Indeed, equality holds if and only if

\begin{align}\label{hesw}
    \nabla^2 w = \frac{\Delta w}{n} g.
\end{align}
Moreover, by Lemma \ref{lemma-tec}, we have
\begin{equation*}
    \frac{1}{f}\langle \nabla w, \nabla f\rangle
    = \frac{1}{n+1}L_f  w = \frac{1}{n+1}\left(\Delta w + \frac{1}{f}\langle \nabla w, \nabla f\rangle\right),
\end{equation*}
which implies
\begin{align}\label{Deltaw}
    \Delta w = \frac{n}{f}\langle \nabla f, \nabla w\rangle.
\end{align}
Combining (\ref{hesw}) and (\ref{Deltaw}), one sees that equality holds in (\ref{eqlemaB}) if and only if
\begin{align*}
    \nabla^2 w = \frac{\Delta w}{n} g = \frac{\langle \nabla f, \nabla w\rangle}{f} g.
\end{align*} This completes the proof. 
\end{proof}

\section{Proof of the Main Results}
\label{SecProofs}

\subsection{Proof of Theorem \ref{thmB}}

\begin{proof}
Let $u$ be an eigenfunction associated with the first nonzero Steklov
eigenvalue $p_1$, that is,
\begin{equation}\label{thmB-eq1}
\begin{cases}
\Delta u=0 & \text{in } M,\\[2mm]
\dfrac{\partial u}{\partial \nu}=p_1u
& \text{on } \partial M.
\end{cases}
\end{equation}

By Proposition \ref{prop-harmonic}, we have
\begin{align}
\int_M f|\nabla^2u|^2\,dV_g
={}&
-\frac{\kappa}{n-1}\int_M|\nabla u|^2\,dV_g
-\frac{R}{n-1}\int_Mf|\nabla u|^2\,dV_g
\nonumber\\
&\quad
+c\int_{\partial M}
\left(\frac{\partial u}{\partial \nu}\right)^2\,dS_g.
\label{thmB-eq2}
\end{align} Using the Steklov boundary condition in \eqref{thmB-eq1}, equation
\eqref{thmB-eq2} becomes
\begin{align}
\int_M f|\nabla^2u|^2\,dV_g
={}&
-\frac{\kappa}{n-1}\int_M|\nabla u|^2\,dV_g
-\frac{R}{n-1}\int_Mf|\nabla u|^2\,dV_g
\nonumber\\
&\quad
+cp_1^2\int_{\partial M}u^2\,dS_g.
\label{thmB-eq3}
\end{align}

On the other hand, since $\Delta u=0$, the divergence theorem yields
\begin{equation*}
\int_M|\nabla u|^2\,dV_g =
\int_M\operatorname{div}(u\nabla u)\,dV_g =
\int_{\partial M}
u\frac{\partial u}{\partial \nu}\,dS_g=
p_1\int_{\partial M}u^2\,dS_g.
\end{equation*} Consequently,
\[
p_1^2\int_{\partial M}u^2\,dS_g
=
p_1\int_M|\nabla u|^2\,dV_g.
\]
Substituting this into \eqref{thmB-eq3}, we obtain
\begin{align}
\int_M f|\nabla^2u|^2\,dV_g
={}&
\left(
cp_1-\frac{\kappa}{n-1}
\right)
\int_M|\nabla u|^2\,dV_g
\nonumber\\
&\quad
-\frac{R}{n-1}
\int_Mf|\nabla u|^2\,dV_g.
\label{thmB-eq4}
\end{align}

Recall that $H=\frac{\kappa}{c},$ and hence \eqref{thmB-eq4} can be rewritten as
\begin{align}
c\left(
p_1-\frac{H}{n-1}
\right)
\int_M|\nabla u|^2\,dV_g
={}&
\int_Mf|\nabla^2u|^2\,dV_g
\nonumber\\
&\quad
+\frac{R}{n-1}
\int_Mf|\nabla u|^2\,dV_g.
\label{thmB-eq5}
\end{align}
Taking into account that $c>0$, $u$ is nonconstant, $f>0$ in $int(M)$,
and $R\geq0$, the right-hand side of \eqref{thmB-eq5} is nonnegative. Thus, one concludes that
\[
p_1\geq\frac{H}{n-1}.
\]

We now analyze the equality case. Suppose that $p_1=\frac{H}{n-1}.$ Then \eqref{thmB-eq5} gives
\[
\int_Mf|\nabla^2u|^2\,dV_g
+
\frac{R}{n-1}
\int_Mf|\nabla u|^2\,dV_g
=0.
\]
Since both terms are nonnegative, $f>0$ in
$int(M)$ and $u$ is not constant, one sees that
\[
\nabla^2u=0
\qquad\text{and}\qquad
R=0.
\]
Notice also that $p_1>0$. Hence equality implies $H>0$, and then
$\kappa=cH>0$. Using that $\nabla^2u=0$, along $\partial M$ we have
\[
\nabla_{\partial M}^2u
=
\nabla^2u\big|_{T\partial M}
-
\frac{\partial u}{\partial\nu}\,\mathbb{II}.
\]
From \eqref{thmB-eq1}, (\ref{segform}), 
and the equality $p_1=\frac{H}{n-1},$ we obtain
\begin{equation}
\label{eqk1234fr}
\nabla_{\partial M}^2u
=
-p_1^2u\,g_{\partial M}.
\end{equation}
Therefore, by Obata's theorem, one deduces that $\partial M$ is isometric to $\mathbb{S}^{n-1}\left(1/p_1\right).$

It remains to show that $M^n$ itself is a Euclidean ball. Since equality implies
$\kappa>0$, we set 
$$
\lambda=\frac{f}{\kappa}.
$$ Whence, it follows that $\mathcal{L}_g^*(\lambda)=g$, so that $(M^n,\,g,\,\lambda)$ is a Miao--Tam
critical metric in the sense of \cite{miao2009volume,Batistaetal2017}. Taking into account that $\partial M$ is isometric to $\mathbb{S}^{n-1}\left(1/p_1\right),$ we have
\begin{equation}
\label{jku765mk01aabb}
    R^{\partial M}= (n-1)(n-2)p_{1}^2=\frac{(n-2)}{(n-1)}H^2.
\end{equation} Combining \eqref{jku765mk01aabb} and \eqref{eqnhu850001a}, and recalling
that $R=0$, we conclude that
$$
R_{nn}=0.
$$ Using again that $R=0$, we have $\mathring{Ric}=Ric$, and hence
$$
\mathring{Ric}(\nabla\lambda,\nabla\lambda)
=
|\nabla\lambda|^2R_{nn}
=
0.
$$
Therefore, Proposition 1 in \cite{Batistaetal2017} implies that
$(M^n,\,g)$ is isometric to the Euclidean ball $B^n\left(1/p_1\right).$

Conversely, for the Euclidean ball
$B^n(r)$, the first nonzero Steklov eigenvalue is $p_1=\frac{1}{r},$
while its boundary has mean curvature $H=\frac{n-1}{r},$ and hence, $p_1=\frac{H}{n-1}.$ So, the proof is concluded. 
\end{proof}

\subsection{Proof of Corollary \ref{corthmB}}
\begin{proof}
Since $R=0,$ (\ref{eqlap}) reduces to

\begin{equation}
\label{eqlapcorthmB}
    \Delta f = -\frac{\kappa n}{n-1}.
\end{equation} Taking into account that $f>0$ in $int(M)$ and $f=0$ on $\partial M,$ the maximum principle implies that $\kappa>0.$ Next, integrating (\ref{eqlapcorthmB}) over $M$ and using the Stokes' formula, we get

\begin{equation}
\label{eqKeyR0}
    Vol(M)=\frac{n-1}{n H} |\partial M|,
\end{equation} where we used that  $H=\frac{\kappa}{|\nabla f|}.$ Therefore, Theorem \ref{thmB} yields

\begin{equation*}
    p_1 \geq \frac{1}{n}\frac{|\partial M|}{Vol(M)},
\end{equation*} which proves the stated inequality. Moreover, equality in \eqref{ineqcorthmB} is equivalent to equality in Theorem~\ref{thmB}. 
\end{proof}

\subsection{Proof of Theorem \ref{thmD}}
\begin{proof}

Let $u$ be an eigenfunction associated with $q_1$ and set $$w=\Delta u.$$ Whence, it follows that $$\Delta w=0\,\,\,\hbox{in  }M\,\,\,\,\,\,\,\,\hbox{and}\,\,\,\,\,\,\,\,\,w=q_1 \frac{\partial u}{\partial \nu}\,\,\,\hbox{on   }\,\,\,\partial M.$$ Thus, using Green's identity for $u$ and $w,$ one obtains that

\begin{eqnarray}
\label{eqthmd231}
    \int_{M} w^2\,dV_g &=& \int_{M}\left(w\Delta u- u\Delta w\right)dV_g\nonumber\\&=& \int_{\partial M} w\frac{\partial u}{\partial \nu} dS_g\nonumber\\&=& q_1 \int_{\partial M}\left(\frac{\partial u}{\partial \nu}\right)^2 dS_g,
\end{eqnarray} where we also used that $u=0$ on $\partial M$ and $\Delta w=0$ in $M.$ 

At the same time, since $w=q_1 \frac{\partial u}{\partial \nu}$ on $\partial M,$ we see that

\begin{equation*}
    \int_{\partial M} w^2 dS_g=q_1^2 \int_{\partial M} \left(\frac{\partial u}{\partial \nu}\right)^2 dS_g. 
\end{equation*} This substituted into (\ref{eqthmd231}) yields

\begin{equation}
\label{eqkey1thmD}
    q_1 \int_{M}w^2 dV_g = \int_{\partial M}w^2 dS_g.
\end{equation}

On the other hand, applying Green's identity to  $f$ and $w^2,$ one obtains that

\begin{eqnarray*}
\int_{M}\left(f\Delta w^2-w^2\Delta f\right)dV_g &=& -\int_{\partial M} w^2 \frac{\partial f}{\partial \nu}\,dS_g\nonumber\\&=& c\int_{\partial M} w^2 \,dS_g,
\end{eqnarray*} where we have used (\ref{eqrepetc}). We now simplify the two interior terms. Indeed, since $w$ is harmonic, we get $\Delta w^2=2|\nabla w|^2,$ and using (\ref{eqlap}), one deduces that
\begin{equation}
\label{eqHHHTHM3e4}
    c\int_{\partial M}w^2 dS_g=2\int_{M} f|\nabla w|^2 dV_g + \frac{n\kappa}{n-1}\int_{M}w^2 dV_g +\frac{R}{n-1}\int_{M}fw^2\,dV_g.
\end{equation} Therefore, since $w\not\equiv 0,$ by (\ref{eqkey1thmD}), we have

\begin{equation*}
    q_1=\frac{n\kappa}{(n-1)c} +\frac{2}{c}\left(\frac{\displaystyle\int_M f|\nabla w|^2 dV_g}{\displaystyle\int_M w^2 dV_g}\right)+\frac{R}{(n-1)c}\left(\frac{\displaystyle\int_{M}fw^2\,dV_g}{\displaystyle\int_M w^2 dV_g}\right),
\end{equation*} which implies that 

\begin{equation}
\label{eqkey2thmD2}
    q_1\geq \frac{n}{n-1}H,
\end{equation} where we have used that $R\geq 0$ and $H=\kappa/c.$ This proves the stated inequality.

We next address the equality case. If equality holds in (\ref{eqkey2thmD2}), since $w\not\equiv 0$ and $f>0$ in $int(M),$ then $R=0$ and $w=d,$ where $d$ is a nonzero constant. From this, it follows that $$\Delta u=d\,\,\,\,\hbox{in}\,\,\,M\,\,\,\,\,\hbox{and}\,\,\,\,\,\,u=0\,\,\,\hbox{on}\,\,\,\partial M.$$ Now, using that $R=0$ and (\ref{eqlap}), one deduces that 

$$\Delta\left(\frac{u}{d}+\frac{(n-1)}{n\kappa}f\right)=0\,\,\,\,\hbox{in}\,\,\,M\,\,\,\,\,\hbox{and}\,\,\,\,\,\,\,\,\frac{u}{d}+\frac{(n-1)}{n\kappa}f=0\,\,\,\hbox{on}\,\,\,\partial M.$$ By maximum principle, we infer $$u=\frac{(n-1)d}{n\kappa}f=\alpha f,$$ as asserted. 

Conversely, if $R=0$ and $u=\alpha f,$ it follows from (\ref{eqlap}) that 

$$\Delta^2 f=0\,\,\,\,\hbox{in  }\,\,\,M\,\,\,\,\,\,\,\,\,\hbox{and}\,\,\,\,\,\,\,\,\,\Delta f=\frac{n\kappa}{(n-1)c}\frac{\partial f}{\partial \nu}\,\,\,\hbox{on  }\,\,\,\partial M.$$ Hence, $f$ is an eigenfunction of the fourth-order Steklov problem corresponding to the eigenvalue \begin{equation}
\label{eqJKey111}
    \frac{n\kappa}{(n-1)c}=\frac{n}{n-1}H.
\end{equation} By the definition of $q_1$ as the first eigenvalue, we infer

$$q_1\leq \frac{n}{n-1}H.$$ This combined with (\ref{eqkey2thmD2}) gives 

\begin{eqnarray}
\label{ek89011thmD}
    q_1=\frac{n}{n-1}H.
\end{eqnarray} So, the proof of Theorem \ref{thmD} is completed.
\end{proof}

\subsection{Proof of Corollary \ref{corthm3_a}}
\begin{proof}
As observed in the proof of Corollary \ref{corthmB}, since $R=0$, we have necessarily $\kappa>0.$ Thus, we may define $\lambda=\frac{f}{\kappa},$ so that
\[
\mathcal L_g^*(\lambda)=g.
\] Hence, it suffices to invoke Proposition 2.5 of Corvino--Eichmair--Miao
\cite{corv2013def} to infer
\[
\frac{|\partial M|}{Vol(M)}
\leq
\frac{n}{\sqrt{(n-1)(n-2)}}
\left(
\frac{1}{|\partial M|}
\int_{\partial M}R^{\partial M}\,dS_g
\right)^{1/2}.
\] Combining this inequality with \eqref{eq1athmD} and squaring both
sides, we obtain 
\[
q_1^2
\leq
\frac{n^2}{(n-1)(n-2)|\partial M|}
\int_{\partial M}R^{\partial M}\,dS_g,
\]
which proves \eqref{ineqThmD}. Finally, since (\ref{eq1athmD}) is an identity, equality in (\ref{ineqThmD}) is equivalent to equality in Proposition~2.5 of
\cite{corv2013def}. Then, equality holds if and only if $(M^n,g)$ is isometric to a Euclidean ball. 
\end{proof}

\subsection{Proof of Theorem \ref{ThmD_k}}

\begin{proof}
By Lemma~\ref{lem:Fichera}, the eigenvalues of the fourth-order
Steklov problem coincide, including multiplicities, with those of the
generalized eigenvalue problem
\begin{equation}
\label{eq:dual-qk}
\int_{\partial M} h\psi\,dS_g
=
q\int_M h\psi\,dV_g,
\qquad
\psi\in\mathcal H(M).
\end{equation}
Consequently, by the Courant--Fischer min--max principle
\cite[p.~12]{Henrot2006}, the $k$-th eigenvalue is characterized by
\begin{equation}
\label{eq:minmax-qk}
q_k
=
\inf_{\substack{E\subset\mathcal H(M)\\ \dim E=k}}
\sup_{0\neq h\in E}
\frac{\displaystyle\int_{\partial M}h^2\,dS_g}
{\displaystyle\int_M h^2\,dV_g}.
\end{equation}

We now use the $V$-static structure to rewrite the Rayleigh quotient
in \eqref{eq:minmax-qk}. To this end, we first consider a smooth harmonic function
$h$. As in the proof of \eqref{eqHHHTHM3e4}, applying Green's identity to
$f$ and $h^2$, and using $f=0$ on $\partial M$, \eqref{eqrepetc},
the harmonicity of $h$, and \eqref{eqlap}, we obtain

\begin{equation}
\label{eq:rayleigh-qk}
c\int_{\partial M}h^2\,dS_g
=
2\int_M f|\nabla h|^2\,dV_g
+
\frac{n\kappa}{n-1}
\int_Mh^2\,dV_g.
\end{equation}
Thus, for every nonzero harmonic function $h$, we have
\begin{equation}
\label{eq:rayleigh-shift-qk}
\frac{\displaystyle\int_{\partial M}h^2\,dS_g}
{\displaystyle\int_Mh^2\,dV_g}
=
\frac{n\kappa}{(n-1)c}
+
\frac{2}{c}\left(
\frac{\displaystyle\int_M f|\nabla h|^2\,dV_g}
{\displaystyle\int_Mh^2\,dV_g}\right).
\end{equation}
The same identity extends to $h\in\mathcal H(M)$ by density of smooth
harmonic functions in $\mathcal H(M)$ and continuity of the trace
operator.

Substituting \eqref{eq:rayleigh-shift-qk} into the min--max
characterization \eqref{eq:minmax-qk}, we get
\begin{align*}
q_k
&=
\inf_{\substack{E\subset\mathcal H(M)\\ \dim E=k}}
\sup_{0\neq h\in E}
\left(
\frac{n\kappa}{(n-1)c}
+
\frac{2}{c}
\frac{\displaystyle\int_M f|\nabla h|^2\,dV_g}
{\displaystyle\int_Mh^2\,dV_g}
\right)\\
&=
\frac{n\kappa}{(n-1)c}
+
\frac{2}{c}\left(
\inf_{\substack{E\subset\mathcal H(M)\\ \dim E=k}}
\sup_{0\neq h\in E}
\frac{\displaystyle\int_M f|\nabla h|^2\,dV_g}
{\displaystyle\int_Mh^2\,dV_g}\right).
\end{align*}
%Here we used that the first term is independent of both $h$ and $E$, and that $2/c>0$. 
Finally, it suffices to use (\ref{eqKeyR0}) and (\ref{eqJKey111}) to achieve $$q_k = \frac{|\partial M|}{Vol(M)} + \frac{2}{|\nabla f|_{\partial M}}\,\left(
\inf_{\substack{E\subset\mathcal H(M)\\ \dim E=k}}
\sup_{0\neq h\in E}
\frac{\displaystyle\int_M f|\nabla h|^2\,dV_g}
{\displaystyle\int_Mh^2\,dV_g}\right).$$ This finishes the proof of the theorem. 
\end{proof}

\subsection{Proof of Theorem \ref{thmC}}

\begin{proof}
Let $w$ be an eigenfunction associated with the first nonzero weighted
eigenvalue $\mu_1$, normalized so that
\[
\int_M wf\,dV_g=0.
\] Thus, we have 
\begin{equation}\label{thmC-eigen}
L_f w=-\mu_1w
\qquad \text{in }\,\, int(M),
\end{equation}
or, equivalently,
\begin{equation}\label{thmC-div}
\operatorname{div}(f\nabla w)=-\mu_1fw.
\end{equation} Since $f=0$ on $\partial M$, the operator
$L_f$ contains singular terms of the form $1/f$ near the
boundary. For this reason, for $\varepsilon>0$ sufficiently small, we
consider the  domains
\[
M_\varepsilon
=
\{x\in M:f(x)\geq\varepsilon\}.
\]
It turns out that $|\nabla f|=c>0$ on $\partial M$ and hence, by continuity, there exists $\varepsilon_0>0$ such that
$|\nabla f|>0$ on $\{0\leq f\leq\varepsilon_0\}.$ Thereby, every
$\varepsilon\in(0,\varepsilon_0)$ is a regular value of $f$, and
$\partial M_\varepsilon=\{f=\varepsilon\}$ is smooth.

We now use Lemma \ref{lemB}, on $M_\varepsilon,$ and \eqref{thmC-eigen} in order to infer

\begin{equation}
\label{thmC-eq1}
\frac{1}{2}L_f |\nabla w|^2
\geq 
\frac{\mu_1^2}{n+1}w^2
+
\left(
\frac{R}{n-1}-\mu_1
\right)|\nabla w|^2 +
\frac{\kappa}{(n-1)f}|\nabla w|^2.
\end{equation} Multiplying \eqref{thmC-eq1} by $f$ and integrating over
$M_\varepsilon$, we get
\begin{align}
\frac{1}{2}
\int_{M_\varepsilon}
f\,L_f |\nabla w|^2\,dV_g
\geq{}&
\frac{\mu_1^2}{n+1}
\int_{M_\varepsilon}fw^2\,dV_g
\nonumber\\
&+
\left(
\frac{R}{n-1}-\mu_1
\right)
\int_{M_\varepsilon}f|\nabla w|^2\,dV_g
\nonumber\\
&+
\frac{\kappa}{n-1}
\int_{M_\varepsilon}|\nabla w|^2\,dV_g.
\label{thmC-eq2}
\end{align}

 On the other hand, since $f\,L_f \psi
=
\operatorname{div}(f\nabla\psi),$ the divergence theorem gives
\[
\int_{M_\varepsilon}
f\,L_f |\nabla w|^2\,dV_g
=
\int_{\partial M_\varepsilon}
f\frac{\partial |\nabla w|^2}
{\partial\nu_\varepsilon}\,dS_g,
\]
where $\nu_\varepsilon$ denotes the outward unit normal to
$\partial M_\varepsilon$. But, $f=\varepsilon$ on
$\partial M_\varepsilon$ and hence, 
\begin{equation}\label{thmC-boundary1}
\int_{M_\varepsilon}
f\, L_f|\nabla w|^2\,dV_g
=
\varepsilon
\int_{\partial M_\varepsilon}
\frac{\partial |\nabla w|^2}
{\partial\nu_\varepsilon}\,dS_g.
\end{equation}

We now justify the limit of the boundary term. Taking into account that $|\nabla f|>0$ in a neighborhood of $\partial M,$ the level
hypersurfaces $\partial M_\varepsilon$ form a smooth family for all sufficiently small $\varepsilon>0,$ and their $(n-1)$-dimensional
volumes remain uniformly bounded as $\varepsilon\to 0$. Moreover, by
the regularity of $w$, one sees that
\[
\frac{\partial |\nabla w|^2}{\partial\nu_\varepsilon}
=
2\nabla^2w(\nabla w,\nu_\varepsilon)
\]
is uniformly bounded in this neighborhood. Consequently,
\begin{equation}\label{thmC-boundary-limit1}
\varepsilon
\int_{\partial M_\varepsilon}
\frac{\partial |\nabla w|^2}
{\partial\nu_\varepsilon}\,dS_g
\longrightarrow0
\qquad\text{as }\varepsilon\to0.
\end{equation}

Now, taking $\varepsilon\to0$ in \eqref{thmC-eq2}, using
\eqref{thmC-boundary1}, \eqref{thmC-boundary-limit1}, and the dominated
convergence theorem, we obtain
\begin{align}
0\geq{}&
\frac{\mu_1^2}{n+1}
\int_Mfw^2\,dV_g
+
\left(
\frac{R}{n-1}-\mu_1
\right)
\int_Mf|\nabla w|^2\,dV_g
\nonumber\\
&+
\frac{\kappa}{n-1}
\int_M|\nabla w|^2\,dV_g.
\label{thmC-eq3}
\end{align}

At the same time, multiplying \eqref{thmC-div} by $w$ and integrating over
$M_\varepsilon$, we achieve
\begin{equation}\label{thmC-eq4}
\mu_1\int_{M_\varepsilon}fw^2\,dV_g
=
\int_{M_\varepsilon}f|\nabla w|^2\,dV_g
-
\varepsilon
\int_{\partial M_\varepsilon}
w\frac{\partial w}{\partial\nu_\varepsilon}\,dS_g.
\end{equation} As above, the regularity of $w$ and the uniform boundedness of the
areas of $\partial M_\varepsilon$ imply
\begin{equation*}
\varepsilon
\int_{\partial M_\varepsilon}
w\frac{\partial w}{\partial\nu_\varepsilon}\,dS_g
\longrightarrow0
\qquad\text{as }\varepsilon\to0.
\end{equation*} Letting $\varepsilon\to0$ in \eqref{thmC-eq4}, we conclude that
\begin{equation}\label{thmC-energy}
\int_Mf|\nabla w|^2\,dV_g
=
\mu_1\int_Mfw^2\,dV_g.
\end{equation} Substituting \eqref{thmC-energy} into \eqref{thmC-eq3}, we therefore obtain
\begin{align}
0\geq{}&
\mu_1
\left(
\frac{R}{n-1}
-
\frac{n}{n+1}\mu_1
\right)
\int_Mfw^2\,dV_g
\nonumber\\
&+
\frac{\kappa}{n-1}
\int_M|\nabla w|^2\,dV_g.
\label{thmC-eq5}
\end{align} Besides, by \eqref{thmC-energy}, one sees that
\begin{equation}\label{thmC-eq6}
\int_M|\nabla w|^2\,dV_g
\geq
\frac{\mu_1}{f_{\max}}
\int_Mfw^2\,dV_g.
\end{equation} Since $\kappa\geq0$, \eqref{thmC-eq5} and
\eqref{thmC-eq6} yield
\[
0
\geq
\mu_1
\left(
\frac{R}{n-1}
-
\frac{n}{n+1}\mu_1
+
\frac{\kappa}{(n-1)f_{\max}}
\right)
\int_Mfw^2\,dV_g.
\] Whence, it follows that
\begin{equation}\label{thmC-final}
\mu_1
\geq
\frac{n+1}{n(n-1)}
\left(
R+\frac{\kappa}{f_{\max}}
\right),
\end{equation} which proves the stated inequality.

We now analyze the equality case. Suppose equality holds in
\eqref{thmC-final}. We first claim that $\kappa=0$. Indeed, if
$\kappa>0$, then equality must also hold in \eqref{thmC-eq6}. Thus,
\[
\int_M
\left(
1-\frac{f}{f_{\max}}
\right)
|\nabla w|^2\,dV_g
=
0.
\] Since $f$ is not constant, one deduces that $\nabla w=0.$ Hence, by (\ref{thmC-eq6}), one sees that $w\equiv0$ on $M$,
which contradicts the choice of $w.$ Therefore, $\kappa=0$ as claimed. 

Thereby, in the equality case, we have
\begin{equation}\label{thmC-equality-mu}
\mu_1
=
\frac{n+1}{n(n-1)}R.
\end{equation}
Moreover, equality must hold in Lemma \ref{lemB}.
Consequently,
\begin{equation}\label{thmC-hessian1}
\nabla^2w
=
\frac{\Delta w}{n}g
=
\frac{\langle\nabla f,\nabla w\rangle}{f}\,g.
\end{equation}
It follows from \eqref{thmC-hessian1} that
\[
L_f w
=
\Delta w
+
\frac{\langle\nabla f,\nabla w\rangle}{f}
=
\frac{n+1}{n}\Delta w,
\] and using \eqref{thmC-eigen}, we obtain
\[
\Delta w
=
-\frac{n\mu_1}{n+1}w.
\] This substituted into \eqref{thmC-hessian1} gives
\begin{equation}\label{thmC-hessian2}
\nabla^2w
=
-\frac{\mu_1}{n+1}wg.
\end{equation}
Furthermore, we have 
\[
\frac{\langle\nabla f,\nabla w\rangle}{f}
=
-\frac{\mu_1}{n+1}w,
\]
and therefore
\begin{equation}\label{thmC-f-w}
\langle\nabla f,\nabla w\rangle
=
-\frac{\mu_1}{n+1}fw.
\end{equation} Since $f=0$ on $\partial M$, we get
\[
\langle\nabla f,\nabla w\rangle=0
\qquad\text{on }\partial M.
\]

On the other hand, it is known that $\nabla f=-c\nu$ on $\partial M,$
where $c=|\nabla f|_{\partial M}>0$. Hence,
\begin{equation}\label{thmC-neumann}
\frac{\partial w}{\partial\nu}=0
\qquad\text{on }\partial M.
\end{equation} Combining \eqref{thmC-hessian2} and \eqref{thmC-neumann}, we obtain
\[
\begin{cases}
\displaystyle
\nabla^2w
+
\frac{\mu_1}{n+1}wg=0
& \text{in }M,\\[2mm]
\displaystyle
\frac{\partial w}{\partial\nu}=0
& \text{on }\partial M.
\end{cases}
\]
Since $w$ is nonconstant, the boundary version of Obata's rigidity
theorem (see \cite{Reilly2}) implies that $(M^n,\,g)$ is isometric to a hemisphere of constant
sectional curvature $K=\frac{\mu_1}{n+1}.$ Whence, it follows from \eqref{thmC-equality-mu} that $(M^n,\,g)$ is isometric to the hemisphere
$\mathbb{S}^n_+
\left(
\sqrt{\frac{n(n-1)}{R}}
\right).$

Conversely, let $(M^n,g)$ be a hemisphere of constant sectional
curvature $K=\frac{R}{n(n-1)}$ with its standard static potential $f$. Then $\kappa=0$. In particular, the coordinate
functions tangent to the equator satisfy
\[
L_f w=-(n+1)K\,w.
\]
Therefore,
\[
\mu_1=(n+1)K
=
\frac{n+1}{n(n-1)}R,
\]
and equality holds in \eqref{thmC-final}. This completes the proof.
\end{proof}

%\begin{remark}
%We emphasize that no Lojasiewicz-type inequality is required in the
%above limiting argument. Indeed, since $|\nabla f|>0$ on $\partial M$,
%the level sets $\{f=\varepsilon\}$ form a smooth family for all
%sufficiently small $\varepsilon>0$, with uniformly bounded area.
%Moreover, $f=\varepsilon$ on $\partial M_\varepsilon$, so the boundary
%terms arising from integration by parts vanish as $\varepsilon\to0$,
%provided the eigenfunction has the required boundary regularity. This
%differs from arguments involving level sets approaching a critical
%level of the potential, where quantitative Lojasiewicz-type estimates
%may be necessary.
%\end{remark}

\subsection{Proof of Corollary \ref{cor-static-weighted}}
\begin{proof}
Since $(M^n,g,f)$ is a static space, we have $\kappa=0.$ Hence, by Proposition
\ref{pthmK1}, we have $\lambda_1=\frac{R}{n-1}.$ Thereby, Theorem \ref{thmC} implies
\[
\mu_1
\geq
\frac{n+1}{n(n-1)}R
=
\frac{n+1}{n}\lambda_1.
\]
Moreover, the equality characterization follows from Theorem \ref{thmC}.
\end{proof}

\end{document}